\documentclass[12pt]{amsart}

\usepackage{amsmath,amssymb,epic,lscape}

\newtheorem{theorem}{Theorem}[section]
\newtheorem{proposition}[theorem]{Proposition}

\theoremstyle{definition}
\newtheorem{definition}[theorem]{Definition}

\theoremstyle{remark}

\numberwithin{equation}{section}
\providecommand{\bysame}{\leavevmode\hbox to3em{\hrulefill}\thinspace}

\def\DJ{{\hbox{D\kern-.8em\raise.15ex\hbox{--}\kern.35em}}}
\def\DJo{$\;$\kern-.4em
    \hbox{D\kern-.8em\raise.15ex\hbox{--}\kern.35em okovi\'{c}}}

\def\Gr{{ G_{\rm BS} }}
\def\al{{\alpha}}

\def\sig{{\sigma}}

\def\bZ{{\mbox{\bf Z}}}

\def\pE{{\mathcal E}}

\renewcommand{\subjclassname}{\textup{2000} Mathematics Subject
Classification }

\begin{document}

\title[Classification of base sequences $BS(n+1,n)$]
{Classification of base sequences $BS(n+1,n)$}

\author[D.\v{Z}. \DJ okovi\'{c}]
{Dragomir \v{Z}. \DJ okovi\'{c}}

\address{Department of Pure Mathematics, University of Waterloo,
Waterloo, Ontario, N2L 3G1, Canada}

\email{djokovic@uwaterloo.ca}



\date{}

\begin{abstract}
Base sequences $BS(n+1,n)$ are quadruples of $\{\pm1\}$-sequences
$(A;B;C;D)$, with $A$ and $B$ of length $n+1$ and $C$ and $D$ of 
length $n$, such that the sum of their nonperiodic autocorrelation 
functions is a $\delta$-function. The base sequence conjecture, 
asserting that $BS(n+1,n)$ exist for all $n$, is stronger than 
the famous Hadamard matrix conjecture.
We introduce a new definition of equivalence for base 
sequences $BS(n+1,n)$, and construct a canonical form.
By using this canonical form, we have enumerated the equivalence classes of $BS(n+1,n)$ for $n \le 30$. As the number of equivalence 
classes grows rapidly (but not monotonically) with $n$, 
the tables in the paper cover only the cases $n\le13$.
\end{abstract}

\maketitle
\subjclassname{ 05B20, 05B30 }
\vskip5mm

\section{Introduction} \label{Uvod}

Base sequences $BS(m,n)$ are quadruples $(A;B;C;D)$ of binary sequences, with $A$ and $B$ of length $m$ and
$C$ and $D$ of length $n$, such that the sum of
their nonperiodic autocorrelation functions is a 
$\delta$-function. In this paper we take $m=n+1$.

Sporadic examples of base sequences $BS(n+1,n)$ have been
constructed by many authors during the last 30 years, see e.g. 
\cite{HCD,KS,KKS,Y} and the survey paper \cite{SY} and
its references. A more systematic approach has been taken by 
the author in \cite{DZ2,DZ4}. The $BS(n+1,n)$ are presently known 
to exist for all $n\le38$ ({\em ibid}) and for Golay numbers 
$n=2^a10^b26^c$, where $a,b,c$ are arbitrary nonnegative integers. 
However the genuine classification of $BS(n+1,n)$ is still lacking.
Due to the important role that these sequences play in various
combinatorial constructions such as that for $T$-sequences,
orthogonal designs, and Hadamard matrices \cite{HCD,SY,KSY},
it is of interest to classify the base sequences of
small length. Our main goal is to provide such classification
for $n\le30$.

In section \ref{BazNiz} we recall  the basic properties of base 
sequences $BS(n+1,n)$. We also recall the quad decomposition 
and our encoding scheme for this particular type of base sequences. 

In section \ref{Ekv} we enlarge the collection of standard 
elementary transformations of $BS(n+1,n)$ by introducing a new 
one. Thus we obtain new notion of equivalence and
equivalence classes. Throughout the paper, the words
``equivalence" and ``equivalence class" are used in this new 
sense. We also introduce the canonical form for base sequences. 
By using it, we are able to compute the representatives of the 
equivalence classes. 

In section \ref{Grupa} we introduce an abstract group, $\Gr$, 
of order $2^{12}$ which acts naturally on all sets $BS(n+1,n)$. 
Its definition depends on the parity of $n$. The orbits of this 
group are just the equivalence classes of $BS(n+1,n)$.

In section \ref{Tablice} we tabulate some of the results of our
computations (those for $n\le13$) giving the list of representatives 
of the equivalence classes of $BS(n+1,n)$. The representatives are 
written in the encoded form which is explained in the next section.
For $n\le8$ we also include the values of the nonperiodic
autocorrelation functions of the four constituent sequences.
We also raise the question of characterizing the binary 
sequences having the same nonperiodic autocorrelation function.
A class of examples is constructed, showing that the question
is interesting.

The column ``Equ'' in Table 1 gives the number of equivalence classes in $BS(n+1,n)$ for $n\le30$. 
The column ``Nor'' gives the number of normal equivalence classes 
(see Section \ref{Tablice} for their definition).

\begin{center}
Table 1: Number of equivalence classes of $BS(n+1,n)$
\begin{tabular}{rrrcrrr} \\ \hline 
\multicolumn{1}{c}{$n$} & \multicolumn{1}{c}{Equ.} 
& \multicolumn{1}{c}{Nor.} & \multicolumn{1}{c}{\quad\quad\quad} 
& \multicolumn{1}{c}{$n$} & \multicolumn{1}{c}{Equ.} 
& \multicolumn{1}{c}{Nor.} \\ \hline
0 & 1 & 1 && 16 & 1721 & 104 \\
1 & 1 & 1 && 17 & 2241 & 0 \\
2 & 1 & 1 && 18 & 1731 & 2 \\
3 & 1 & 1 && 19 & 4552 & 2 \\
4 & 3 & 2 && 20 & 3442 & 72 \\
5 & 4 & 1 && 21 & 3677 & 0 \\
6 & 5 & 0 && 22 & 15886 & 0 \\
7 & 17 & 6 && 23 & 6139 & 0 \\
8 & 27 & 14 && 24 & 10878 & 0 \\
9 & 44 & 4 && 25 & 19516 & 4 \\
10 & 98 & 10 && 26 & 10626 & 4 \\
11 & 84 & 3 && 27 & 22895 & 0 \\
12 & 175 & 8 && 28 & 31070 & 0 \\
13 & 475 & 5 && 29 & 18831 & 2 \\
14 & 331 & 0 && 30 & 19640 & 0 \\
15 & 491 & 2 && 31 & ? & 0 \\
\hline
\end{tabular} \\
\end{center}

\section{Quad decomposition and the encoding scheme} \label{BazNiz}

We denote finite sequences of integers by capital letters. 
If, say, $A$ is such a sequence of length $n$ then we denote 
its elements by the corresponding lower case letters. Thus 
$$ A=a_1,a_2,\ldots,a_n. $$
To this sequence we associate the polynomial
$$ A(x)=a_1+a_2x+\cdots+a_nx^{n-1} , $$
viewed as an element of the Laurent polynomial ring 
$\bZ[x,x^{-1}]$. (As usual, $\bZ$ denotes the ring of integers.)
The {\em nonperiodic autocorrelation function} $N_A$ of $A$ is 
defined by:
$$ N_A(i)=\sum_{j\in\bZ} a_ja_{i+j},\quad i\in\bZ, $$
where $a_k=0$ for $k<1$ and for $k>n$. Note that
$N_A(-i)=N_A(i)$ for all $i\in\bZ$ and $N_A(i)=0$ for $i\ge n$.
The {\em norm} of $A$ is the Laurent polynomial 
$N(A)=A(x)A(x^{-1})$. We have
$$ N(A)=\sum_{i\in\bZ} N_A(i) x^i . $$
The negation, $-A$, of $A$ is the sequence
$$ -A=-a_1,-a_2,\ldots,-a_n. $$
The {\em reversed} sequence $A'$ and the {\em alternated} sequence
$A^*$ of the sequence $A$ are defined by
\begin{eqnarray*}
A' &=& a_n,a_{n-1},\ldots,a_1 \\
A^* &=& a_1,-a_2,a_3,-a_4,\ldots,(-1)^{n-1}a_n.
\end{eqnarray*}
Observe that $N(-A)=N(A')=N(A)$ and $N_{A^*}(i)=(-1)^i N_A(i)$
for all $i\in\bZ$. By $A,B$ we denote the concatenation of the 
sequences $A$ and $B$.

A {\em binary sequence} is a sequence whose terms belong to 
$\{\pm1\}$. When displaying such sequences, we shall often
write $+$ for $+1$ and $-$ for $-1$.
The {\em base sequences} consist of four binary sequences 
$(A;B;C;D)$, with $A$ and $B$ of length $m$ 
and $C$ and $D$ of length $n$, such that 
$$ N(A)+N(B)+N(C)+N(D)=2(m+n). $$
Thus, for $i\ne0$, we have
\begin{equation} \label{KorNula}
N_A(i)+N_B(i)+N_C(i)+N_D(i)=0.
\end{equation}

We denote by $BS(m,n)$ the set of such base sequences with
$m$ and $n$ fixed. From now on we shall consider only
the case $m=n+1$.

Let $(A;B;C;D) \in BS(n+1,n)$. For convenience we fix the 
following notation. For $n$ even (odd) we set $n=2m$ ($n=2m+1$). 
We decompose the pair $(A;B)$ into quads
$$ \left[ \begin{array}{ll} a_i & a_{n+2-i} \\ 
b_i & b_{n+2-i} \end{array} \right],\quad i=1,2,\ldots,
\left[ \frac{n+1}{2} \right], $$ 
and, if $n$ is even, the central column
$ \left[ \begin{array}{l} a_{m+1} \\ b_{m+1} \end{array} \right]. $
Similar decomposition is valid for the pair $(C;D)$.

Recall the following basic and well-known property
\cite[Theorem 1]{KKS}.
\begin{theorem} \label{KKS-teorema}
For $(A;B;C;D)\in BS(n+1,n)$, the sum of the four quad entries is $2 \pmod{4}$ for the first quad of the pair $(A;B)$ and is $0 \pmod{4}$ for all other quads of $(A;B)$ and also for all quads of the pair $(C;D)$.
\end{theorem}

Thus there are 8 possibilities for the first quad of the
pair $(A;B)$:
\begin{center}
\begin{eqnarray*}
1'=\left[ \begin{array}{ll} - & + \\ + & + \end{array} \right],\quad 
2'=\left[ \begin{array}{ll} + & - \\ + & + \end{array} \right],\quad 
3'=\left[ \begin{array}{ll} + & + \\ + & - \end{array} \right],\quad 
4'=\left[ \begin{array}{ll} + & + \\ - & + \end{array} \right], \\
5'=\left[ \begin{array}{ll} + & - \\ - & - \end{array} \right],\quad 
6'=\left[ \begin{array}{ll} - & + \\ - & - \end{array} \right],\quad 
7'=\left[ \begin{array}{ll} - & - \\ - & + \end{array} \right],\quad 
8'=\left[ \begin{array}{ll} - & - \\ + & - \end{array} \right].
\end{eqnarray*}
\end{center}
These eight quads occur in the study of Golay sequences
(see e.g. \cite{DZ1}) and we refer to them as the {\em Golay quads}.

There are also 8 possibilities for each of the remaining quads 
of $(A;B)$ and all quads of $(C;D)$:
\begin{center}
\begin{eqnarray*}
1=\left[ \begin{array}{ll} + & + \\ + & + \end{array} \right],\quad 
2=\left[ \begin{array}{ll} + & + \\ - & - \end{array} \right],\quad 
3=\left[ \begin{array}{ll} - & + \\ - & + \end{array} \right],\quad 
4=\left[ \begin{array}{ll} + & - \\ - & + \end{array} \right], \\
5=\left[ \begin{array}{ll} - & + \\ + & - \end{array} \right],\quad 
6=\left[ \begin{array}{ll} + & - \\ + & - \end{array} \right],\quad 
7=\left[ \begin{array}{ll} - & - \\ + & + \end{array} \right],\quad 
8=\left[ \begin{array}{ll} - & - \\ - & - \end{array} \right].
\end{eqnarray*}
\end{center}
We shall refer to these eight quads as the {\em BS-quads}.
We say that a BS-quad is {\em symmetric} if its two columns are
the same, and otherwise we say that it is {\em skew}. The 
quads $1,2,7,8$ are symmetric and $3,4,5,6$ are skew. We say that
two quads have the {\em same symmetry type} if they are both
symmetric or both skew.

There are 4 possibilities for the central column:
$$
0=\left[ \begin{array}{l} + \\ + \end{array} \right],\quad
1=\left[ \begin{array}{l} + \\ - \end{array} \right],\quad
2=\left[ \begin{array}{l} - \\ + \end{array} \right],\quad
3=\left[ \begin{array}{l} - \\ - \end{array} \right].
$$

We encode the pair $(A;B)$ by the
symbol sequence
\begin{equation} \label{simb-p}
p_1p_2 \ldots p_m p_{m+1},
\end{equation}
where $p_i$ is the label of the $i$th quad except in the case 
where $n$ is even and $i=m+1$ in which case $p_{m+1}$ is the
label of the central column. 

Similarly, we encode the pair $(C;D)$ by the symbol sequence
\begin{equation} \label{simb-q}
q_1q_2 \ldots q_m \quad  \text{respectively} 
\quad q_1q_2 \ldots q_m q_{m+1}
\end{equation}
when $n$ is even respectively odd. Here $q_i$ is the label of the 
$i$th quad for $i\le m$ and $q_{m+1}$ is the label of the central 
column (when $n$ is odd).

\section{The equivalence relation} \label{Ekv}

We start by defining five types of elementary transformations 
of base sequences $BS(n+1,n)$. These elementary transformations 
include the standard ones, as described in \cite{KKS} and 
\cite{DZ2}. But we also  introduce one additional elementary 
transformation, see item (T4) below, which made its first 
appearance in \cite{DZ3} in the context of near-normal 
sequences. The quad notation was instrumental in the discovery
of this new elementary operation.

The {\em elementary transformations} of $(A;B;C;D)\in BS(n+1,n)$ 
are the following:

(T1) Negate one of the sequences $A;B;C;D$.

(T2) Reverse one of the sequences $A;B;C;D$.

(T3) Interchange the sequences $A;B$ or $C;D$.

(T4) Replace the pair $(C;D)$ with the pair 
$(\tilde{C};\tilde{D})$ which is defined as follows:
If (\ref{simb-q}) is the encoding of $(C;D)$, then 
the encoding of $(\tilde{C};\tilde{D})$ is 
$\tau(q_1)\tau(q_2)\cdots\tau(q_m) q_{m+1}$
or $\tau(q_1)\tau(q_2)\cdots\tau(q_m)$
depending on whether $n$ is even or odd,
where $\tau$ is the transposition $(45)$.
In other words, the encoding of $(\tilde{C};\tilde{D})$ is 
obtained from that of $(C;D)$ by replacing each quad symbol 4 
with the symbol 5, and vice versa. (We verify below that $N_{\tilde{C}}+N_{\tilde{D}}=N_C+N_D$.)

(T5) Alternate all four sequences $A;B;C;D$.

In order to justify (T4) one has to verify that  $N_{\tilde{C}}+N_{\tilde{D}}=N_C+N_D$. For that purpose let us fix two quads $q_k$ and $q_{k+i}$ and consider their contribution $\delta_i$ to $N_C(i)+N_D(i)$. We claim that $\delta_i$ is equal to the contribution $\tilde{\delta}_i$ of $\tau(q_k)$ and $\tau(q_{k+i})$ to $N_{\tilde{C}}(i)+N_{\tilde{D}}(i)$. If neither $q_k$ nor $q_{k+i}$ belongs to $\{4,5\}$, then $\tau(q_k)=q_k$ and $\tau(q_{k+i})=q_{k+i}$ and so $\delta_i=\tilde{\delta}_i$. If $q\in\{4,5\}$ then $\tau(q)$ is the negation of $q$. Hence if $\{q_k,q_{k+i}\}\subseteq\{4,5\}$ then again $\delta_i=\tilde{\delta}_i$. Otherwise, say $q_k\in\{4,5\}$ while $q_{k+i}\notin\{4,5\}$, and it is easy to verify that $\delta_i=0=\tilde{\delta}_i$. The pairs $q_k$ and $q_{n+1-k-i}$ also make a contribution to $N_C(i)+N_D(i)$, but they can be treated in the same manner. Finally, if $n$ is odd then the pair $(C;D)$ also has a central column with label $q_{m+1}$. In that case, if $k=m+1-i$ and $q_k\in\{4,5\}$, the contribution of $q_k$ and $q_{m+1}$ to $N_C(i)+N_D(i)$ is 0. This completes the verification.

We say that two members of $BS(n+1,n)$ are {\em equivalent} if 
one can be transformed to the other by applying a finite 
sequence of elementary transformations. 
One can enumerate the equivalence classes by finding suitable
representatives of the classes.
For that purpose we introduce the canonical form.

\begin{definition} \label{KanFor}
Let $S=(A;B;C;D) \in BS(n+1,n)$ and let (\ref{simb-p})
respectively (\ref{simb-q}) be the encoding of the pair
$(A;B)$ respectively $(C;D)$. We say that $S$ is in the 
{\em canonical form} if the following eleven conditions hold:

(i) $p_1=3'$, $p_2\in\{6,8\}$ for $n$ even and $p_2\in\{1,6\}$ 
for $n$ odd.

(ii) The first symmetric quad (if any) of $(A;B)$ is 1 or 8.

(iii) If $n$ is even and $p_i\in\{3,4,5,6\}$ for $2\le i\le m$ 
then $p_{m+1}\in\{0,3\}$.

(iv) The first skew quad (if any) of $(A;B)$ is 3 or 6.

(v) $q_1=1$ for $n$ even and $q_1\in\{1,6\}$ for $n$ odd.

(vi) The first symmetric quad (if any) of $(C;D)$ is 1.

(vii) The first skew quad (if any) of $(C;D)$ is 6.

(viii) If $i$ is the least index such that $q_i\in\{2,7\}$ 
then $q_i=2$.

(ix) If $i$ is the least index such that 
$q_i\in\{4,5\}$ then $q_i=4$.

(x)  If $n$ is odd and $q_i \in \{1,3,6,8\}$, 
$\forall i\le m$, then $q_{m+1}\ne2$.

(xi)  If $n$ is odd and $q_i \in \{3,4,5,6\}$, 
$\forall i\le m$, then $q_{m+1}=0$.
\end{definition}

We can now prove that each equivalence class has a member 
which is in the canonical form. The uniqueness of this member 
will be proved in the next section.

\begin{proposition} \label{Klase}
Each equivalence class $\pE\subseteq BS(n+1,n)$ has at least 
one member having the canonical form.
\end{proposition}
\begin{proof}
Let $S=(A;B;C;D)\in\pE$ be arbitrary and let (\ref{simb-p}) 
respectively (\ref{simb-q}) be the encoding of $(A;B)$ 
respectively $(C;D)$. By applying the first three types of 
elementary transformations and by Theorem \ref{KKS-teorema} we can assume that $p_1=3'$ and $c_1=d_1=+1$. By Theorem \ref{KKS-teorema}, $q_1\in\{1,6\}$. If $n$ is even and $q_1=6$ we apply the elementary transformation (T5). Thus we may assume that $p_1=3'$ and that the condition (v) for the canonical form is satisfied. To satisfy the conditions (ii) and (iii), replace $B$ with $-B'$ (if necessary). To satisfy the condition (iv), replace $A$ with $A'$ (if necessary). 

We now modify $S$ in order to satisfy the second part of condition (i). Note that $p_2$ is a BS-quad by Theorem \ref{KKS-teorema}. If $p_2=6$ there is nothing to do.

Assume that the quad $p_2$ is symmetric. By (ii), $p_2\in\{1,8\}$.
From equation (\ref{KorNula}) for $i=n-1$,
we deduce that $p_2=8$ if $q_1=1$ and $p_2=1$ if $q_1=6$. Note that if 
$n$ is even, then $q_1=1$ by (v). If $n$ is odd and $q_1=1$, we switch 
$A$ and $B$ and apply the elementary transformation (T5). After this 
change we still have $p_1=3'$, $q_1=1$, the conditions (ii), 
(iii) and (iv) remain satisfied, and moreover $p_2=6$.

Now assume that $p_2$ is skew. In view of (iv), we may assume that 
$p_2=3$. Then the argument above based on the equation 
(\ref{KorNula}) shows that $q_1=6$, and so $n$ must be odd. After 
applying the elementary transformation (T5), we obtain that $p_2=1$. 
Hence the condition (i) is fully satisfied.

To satisfy (vi), in view of (v) we may assume that $n$ is odd and
$q_1=6$. If the first symmetric quad in $(C;D)$ is 2 respectively 7,
we reverse and negate $C$ respectively $D$. If it is 8, we reverse and negate both $C$ and $D$. Now the first symmetric quad will be 1.

To satisfy (vii), (if necessary) reverse $C$ or $D$, or both. To satisfy (viii), (if necessary) interchange $C$ and $D$. Note that in this process we do not violate the previously established  properties. To satisfy (ix), (if necessary) apply the elementary transformation (T4). To satisfy (x), switch $C$ and $D$ (if necessary). To satisfy (xi), (if necessary) replace $C$ with $-C'$ or $D$ with $-D'$, or both.

Hence $S$ is now in the canonical form.

\end{proof}

\section{The symmetry group of $BS(n+1,n)$} \label{Grupa}

We shall construct a group $\Gr$ of order $2^{12}$ which acts on
$BS(n+1,n)$. Our (redundant) generating set for 
$\Gr$ will consist of 12 involutions. Each of these generators 
is an elementary transformation, and we use 
this information to construct $\Gr$, i.e., to impose the defining 
relations. We denote by $S=(A;B;C;D)$ an aritrary member of 
$BS(n+1,n)$.

To construct $\Gr$, we start with an elementary abelian group $E$ 
of order $2^8$ with generators $\nu_i,\rho_i$, $i\in\{1,2,3,4\}$.
It acts on $BS(n+1,n)$ as follows:
\begin{eqnarray*}
&& \nu_1S=(-A;B;C;D),\quad \rho_1S=(A';B;C;D), \\
&& \nu_2S=(A;-B;C;D),\quad \rho_2S=(A;B';C;D), \\
&& \nu_3S=(A;B;-C;D),\quad \rho_3S=(A;B;C';D), \\
&& \nu_4S=(A;B;C;-D),\quad \rho_4S=(A;B;C;D'),
\end{eqnarray*}
i.e., $\nu_i$ negates the $i$th sequence of $S$ and $\rho_i$ 
reverses it.

Next we introduce two commuting involutory generators $\sig_1$ 
and $\sig_2$. We declare that $\sig_1$ commutes with
$\nu_3,\nu_4,\rho_3,\rho_4$, and $\sig_2$ commutes with
$\nu_1,\nu_2,\rho_1,\rho_2$, and that
$$
\sig_1\nu_1=\nu_2\sig_1,\quad
\sig_1\rho_1=\rho_2\sig_1,\quad
\sig_2\nu_3=\nu_4\sig_2,\quad
\sig_2\rho_3=\rho_4\sig_2.
$$
The group $H=\langle E,\sig_1,\sig_2 \rangle$ is 
the direct product of two isomorphic groups of order 32:
$$ H_1= \langle \nu_1,\rho_1,\sig_1 \rangle, \quad
\text{and} \quad H_2=\langle \nu_3,\rho_3,\sig_2 \rangle . $$
The action of $E$ on $BS(n+1,n)$ extends to $H$ by defining:
$$ \sig_1S=(B;A;C;D),\quad \sig_2S=(A;B;D;C). $$

We add a new generator $\theta$ which commutes elementwise with $H_1$, commutes with $\nu_3\rho_3,\nu_4\rho_4$ and $\sig_2$, and satisfies
$\theta\rho_3=\rho_4\theta$. Let us denote this enlarged group 
by $\tilde{H}$. It has the direct product decomposition 
$$ \tilde{H}=\langle H,\theta \rangle = H_1 \times \tilde{H}_2,$$
where the second factor is itself direct product of two 
copies of the dihedral group $D_8$ of order 8:
$$ \tilde{H}_2=\langle \rho_3,\rho_4,\theta \rangle \times
\langle \nu_3\rho_3,\nu_4\rho_4,\theta\sig_2 \rangle. $$
The action of $H$ on $BS(n+1,n)$ extends to $\tilde{H}$ by 
letting $\theta$ act as the elementary transformation (T4).

Finally, we define $\Gr$ as the semidirect product of 
$\tilde{H}$ and the group of order 2 with generator $\al$. 
By definition, $\al$ commutes with each $\nu_i$ and satisfies:
\begin{eqnarray*}
&& \al\rho_i\al=\rho_i\nu_i^n,\ i=1,2; \\
&& \al\rho_j\al=\rho_j\nu_j^{n-1},\ j=3,4; \\
&& \al\theta\al=\theta\sig_2^{n-1}.
\end{eqnarray*}
The action of $ \tilde{H}$ on $BS(n+1,n)$ extends to $\Gr$ by 
letting $\al$ act as the elementary transformation (T5), i.e., 
we have
$$ \al S=(A^*;B^*;C^*;D^*). $$ 

We point out that the definition of the subgroup $\tilde{H}$ is independent of $n$ and its action on $BS(n+1,n)$ has a quad-wise 
character. By this we mean that the value of a particular 
quad, say $p_i$, of $S\in BS(n+1,n)$ and $h\in\tilde{H}$ determine uniquely the quad $p_i$ of $hS$. In other words
$\tilde{H}$ acts on the Golay quads, the BS-quads and the
set of central columns such that the encoding of $hS$ is
given by the symbol sequences
$$ h(p_1)h(p_2)\ldots h(p_{m+1}) \quad \text{and} \quad
h(q_1)h(q_2)\ldots \, .$$
On the other hand the full group $\Gr$ has neither of these two 
properties.

An important feature of the action of $\tilde{H}$ on the
BS-quads is that it preserves the symmetry type of the quads.

The following proposition follows immediately from the construction of $\Gr$ and the description of its action on 
$BS(n+1,n)$.

\begin{proposition} \label{Orbite}
The orbits of $\Gr$ in $BS(n+1,n)$ are the same as the
equivalence classes.
\end{proposition}

The main tool that we use to enumerate the equivalence classes of $BS(n+1,n)$ is the following theorem.

\begin{theorem} \label{Glavna}
For each equivalence class $\pE\subseteq BS(n+1,n)$ there is a 
unique $S=(A;B;C;D)\in\pE$ having the canonical form.
\end{theorem}
\begin{proof}
In view of Proposition \ref{Klase}, we just have to prove the 
uniqueness assertion. Let
$$
S^{(k)}=(A^{(k)};B^{(k)};C^{(k)};D^{(k)})\in\pE,\quad (k=1,2)
$$
be in the canonical form. We have to prove that in fact 
$S^{(1)}=S^{(2)}$.

By Proposition \ref{Orbite}, we have $gS^{(1)}=S^{(2)}$ for some $g\in\Gr$. We can write $g$ as $g=\al^s h$ where $s\in\{0,1\}$ and $h\in\tilde{H}$. Let $p_1^{(k)}p_2^{(k)}\ldots p_{m+1}^{(k)}$ be the encoding of the pair $(A^{(k)};B^{(k)})$ and $q_1^{(k)}q_2^{(k)}\ldots$ the encoding of the pair $(C^{(k)};D^{(k)})$. The symbols (i-xi)
will refer to the corresponding conditions of Definition \ref{KanFor}. Observe that $p_1^{(1)}=p_1^{(2)}=3'$ by (i).

We prove first preliminary claims (a-c). 

(a): $q_1^{(1)}=q_1^{(2)}$. 

For $n$ even see (v). Let $n$ be odd. 
When we apply the generator $\al$ to any $S\in BS(n+1,n)$, we do not change the first quad of $(C;D)$. 
It follows that the quads $q_1^{(1)}$ and $q_1^{(2)}$ have the same 
symmetry type. The claim now follows from (v).

(b): $g\in\tilde{H}$, i.e., $s=0$. 

Assume first that $n$ is even. By (v), $q_1^{(1)}=q_1^{(2)}=1$. 
For any $S\in\pE$ the first quad of
$(C;D)$ in $S$ and in $\al S$ have different symmetry types.
As the quad $h(1)$ is symmetric, the equality 
$\al^s hS^{(1)}=S^{(2)}$ forces $s$ to be 0. 
Assume now that $n$ is odd. Then for any $S\in\pE$ the second quad of
$(A;B)$ in $S$ and in $\al S$ have different symmetry types.
Recall that $q_1^{(1)}=q_1^{(2)}\in\{1,6\}$ and that (see (i)) 
$p_2^{(1)}$ and $p_2^{(2)}$ belong to $\{1,6\}$. 
From (\ref{KorNula}), with $i=n-1$, we deduce that 
$p_2^{(k)}\ne q_1^{(k)}$ for $k=1,2$. 
We conclude that $p_2^{(1)}=p_2^{(2)}$. The claim now follows from the fact that $h$ preserves while $\al$ alters the symmetry type of the quad $p_2$.

As an immediate consequence of (b), we point out that
a quad $p_i^{(1)}$ is symmetric iff $p_i^{(2)}$ is,
and the same is true for the quads $q_i^{(1)}$ and $q_i^{(2)}$.

(c): $p_2^{(1)}=p_2^{(2)}$. 

This was already proved above in the case when $n$ is odd. In general, the claim follows from (b) and the equality 
$h(p_2^{(1)})=p_2^{(2)}$.
Observe that each of the sets $\{6,8\}$ and $\{1,6\}$ consists 
of one symmetric and one skew quad and that $h$ preserves the
symmetry type of quads.

Recall that $\tilde{H}=H_1\times \tilde{H}_2$. Since $s=0$ we have $g=h=h_1h_2$ with $h_1\in H_1$ and $h_2\in \tilde{H}_2$.
Consequently, $h_1(p_i^{(1)})=p_i^{(2)}$ and
$h_2(q_i^{(1)})=q_i^{(2)}$ for all $i$s.

We shall now prove that $A^{(1)}=A^{(2)}$ and $B^{(1)}=B^{(2)}$. 
Since $p_1^{(1)}=p_1^{(2)}=3'$, the equality
$h_1(p_1^{(1)})=p_1^{(2)}$ implies that $h_1(3')=3'$. Thus 
$h_1=\rho_1^e (\nu_2 \rho_2)^f$ for some $e,f\in\{0,1\}$.

Assume first that $p_2^{(1)}$ is symmetric. By (ii), 
$p_2^{(1)}\in\{1,8\}$. Then $h_1(p_2^{(1)})=p_2^{(2)}=p_2^{(1)}$ implies that $f=0$. Hence, $h_1=\rho_1^e$ and so $B^{(1)}=B^{(2)}$.
If all quads $p_i^{(1)}$, $i\ne1$, are symmetric then
also $A^{(2)}=h_1A^{(1)}=A^{(1)}$. Otherwise let $i$ be the least index for which the quad $p_i^{(1)}$ is skew. Since $B^{(1)}=B^{(2)}$
and $p_i^{(1)}$ is 3 or 6 (see (iv)), we infer
that $e=0$. Hence $h_1=1$ and so $A^{(1)}=A^{(2)}$. 

Now assume that $p_2^{(1)}$ is skew. By (ii), 
$p_2^{(1)}=6$. Then $h_1(p_2^{(1)})=p_2^{(2)}=p_2^{(1)}$ implies that $e=0$. Thus $h_1=(\nu_2 \rho_2)^f$ and so $A^{(1)}=A^{(2)}$.
If all quads $p_i^{(1)}$, $i\ne1$, are skew then by invoking 
the condition (iii) we deduce that $f=0$ and so $B^{(1)}=B^{(2)}$. 
Otherwise let $i$ be the least index for 
which the quad $p_i^{(1)}$ is symmetric. Since $A^{(1)}=A^{(2)}$
and $p_i^{(1)}$ is 1 or 8 (see (ii)), we infer
that $f=0$. Hence $h_1=1$ and so $B^{(1)}=B^{(2)}$. 

It remains to prove that $C^{(1)}=C^{(2)}$ and $D^{(1)}=D^{(2)}$. 
We set $Q=\{q_i^{(1)}:1\le i\le m\}$.
By (v) and the claim (a) we have $q_1^{(1)}=q_1^{(2)}\in\{1,6\}$.

We first consider the case $q_1^{(1)}=q_1^{(2)}=6$ which occurs only
for $n$ odd. Then $h_2(6)=6$ and so 
$h_2\in\langle\nu_3\rho_3,\nu_4\rho_4,\theta,\sig_2\rangle$.
It follows that $h_2(3)=3$.

If some $q\in Q$ is symmetric, let $i$ be the least
index such that $q_i^{(1)}$ is symmetric. Then (vi) implies that
$q_i^{(1)}=q_i^{(2)}=1$. Thus $h_2$ must fix the quad 1. As the stabilizer of the quad 1 in 
$\langle\nu_3\rho_3,\nu_4\rho_4,\theta,\sig_2\rangle$
is $\langle\theta,\sigma_2\rangle$, we infer that $h_2$ must also fix the quad 8. Similarly, if $2\in Q$ then (viii) implies that $h_2$ fixes 2 and 7. If $4\in Q$ then (ix) implies that $h_2$ fixes 4 and 5. These facts imply that $h_2$ fixes all quads in $Q$, i.e., $q_i^{(1)}=q_i^{(2)}$ for all $i\le m$. It remains to show that, for odd $n$, $q_{m+1}^{(1)}=q_{m+1}^{(2)}$. If $Q\subseteq\{3,4,5,6\}$, this follows from (xi). Otherwise $Q$ contains a symmetric quad and so $h_2\in\langle\theta,\sigma_2\rangle$. If $Q\not\subseteq\{1,3,6,8\}$ then $Q$ contains one of the quads 2,4,5 or 7. Since $h_2$ fixes all quads in $Q$, we infer that $h_2\in\langle\theta\rangle$, and so $q_{m+1}^{(1)}=q_{m+1}^{(2)}$. If $Q\subseteq\{1,3,6,8\}$, the equality $q_{m+1}^{(1)}=q_{m+1}^{(2)}$ follows from (x).

Finally, we consider the case $q_1^{(1)}=q_1^{(2)}=1$. Since
$h_2(q_1^{(1)})=q_1^{(2)}$, $h_2\in\langle \rho_3,\rho_4,\theta,\sig_2 \rangle$. Hence $h_2$ fixes the quads 1 and 8.

If some $q\in Q$ is skew, then (vii) implies that $h_2$ fixes the quads 3 and 6. If $2\in Q$ then (viii) implies that $h_2$ fixes the quads 2 and 7. If $4\in Q$ then (ix) implies that $h_2$ fixes the quads 4 and 5. These facts imply that $h_2$ fixes all quads in $Q$. If $n$ is odd,
then we invoke the conditions (x) and (xi) to conclude that $h_2$
also fixes the central column of $(C^{(1)};D^{(1)})$. Hence
$C^{(1)}=C^{(2)}$ and $D^{(1)}=D^{(2)}$ also in this case. 
\end{proof}

\section{Representatives of the equivalence classes} \label{Tablice}

We have computed a set of representatives for the equivalence
classes of base sequences $BS(n+1,n)$ for all $n\le30$. Due to
their excessive size, we tabulate these sets only for $n\le13$. 
Each representative is given in the canonical form
which is made compact by using our standard encoding.
The encoding is explained in detail in Section \ref{BazNiz}. 

As an example, the base sequences
\begin{eqnarray*}
A &=& +,+,+,+,-,-,+,-,+; \\
B &=& +,+,+,-,+,+,+,-,-; \\
C &=& +,+,-,-,+,-,-,+; \\
D &=& +,+,+,+,-,+,-,+; \\
\end{eqnarray*}
are encoded as $3'6142;\, 1675$. In the tables we write 0 
instead of $3'$. This convention was used in our previous 
papers on this and related topics.

This compact notation is used primarily in order to save space, 
but also to avoid introducing errors during decoding. For each 
$n$, the representatives are listed in the lexicographic order 
of the symbol sequences (\ref{simb-p}) and (\ref{simb-q}).

In Table 2 we list the codes for the representatives of the
equivalence classes of $BS(n+1,n)$ for $n\le8$. This table also 
records the values $N_X(k)$ of the nonperiodic autocorrelation 
functions for $X\in\{A,B,C,D\}$ and $k\ge0$. For instance let us consider the first item in the list of base sequences $BS(8,7)$ given in Table 2. The base sequences are encoded in the first column as 0165; 6123. The first part 0165 encodes the pair $(A;B)$, and the second part 6123 the pair $(C;D)$. The function $N_A$, at the points
$0,1,\ldots,7$, takes the values $8,-1,2,-1,0,1,2,1$ listed in the second column. Just below these values one finds the values of $N_B$  at the same points. In the third column we list likewise the values 
of $N_C$ and $N_D$ at the points $i=0,1,\ldots,6$.

Tables 3-7 contain only the list of codes for the representatives of the equivalence classes of $BS(n+1,n)$ for $9\le n\le 13$. 

Let us say that the base sequences $S=(A;B;C;D) \in BS(n+1,n)$ 
are {\em normal} respectively {\em near-normal} if
$b_i=a_i$ respectively $b_i=(-1)^{i-1}a_i$ for all $i\le n$.
We denote by $NS(n)$ respectively $NN(n)$ the set of all
normal respectively near-normal sequences in $BS(n+1,n)$.
Let us say also that an equivalence class $\pE\subseteq BS(n+1,n)$
is {\em normal} respectively {\em near-normal} if
$\pE \cap NS(n)$ respectively $\pE \cap NN(n)$ is not void.
Our canonical form has been designed so that if $\pE$ is normal
then its canonical representative $S$ belongs to $NS(n)$. 
The analogous statement for near-normal classes is false.
It is not hard to recognize which representatives $S$ in
our tables are normal sequences. Let (\ref{simb-p}) be the
encoding of the pair $(A;B)$. Then $S\in NS(n)$ iff all the
quads $p_i$, $i\ne1$, belong to $\{1,3,6,8\}$ and, in the case
when $n=2m$ is even, the central column symbol $p_{m+1}$ is 0 or 3.

It is an interesting question to find the necessary and sufficient
conditions for two binary sequences to have the same norm. The group of order four generated by the negation and reversal operations
acts on binary sequences. We say that two binary sequences are
{\em equivalent} if they belong to the same orbit.
Note that the equivalent binary sequences have the same norm.
However the converse is false. Here is a counter-example
which occurs in Table 2 for the case $n=8$. The base sequences 
15 and 16 differ only in their first sequences, which we denote
here by $U$ and $V$ respectively:
\begin{eqnarray*}
U &=& +\,+\,-\,-\,-\,+\,-\,-\,+ \, , \\
V &=& +\,+\,-\,+\,+\,-\,-\,-\,+ \, .
\end{eqnarray*}
Their associated polynomials are
\begin{eqnarray*}
U(x) &=& 1+x-x^2-x^3-x^4+x^5-x^6-x^7+x^8 \, , \\
V(x) &=& 1+x-x^2+x^3+x^4-x^5-x^6-x^7+x^8.
\end{eqnarray*}
It is obvious that $U$ and $V$ are not equivalent in the above sense. On the other hand, from the factorizations $U(x)=p(x)q(x)$ and 
$V(x)=p(x)r(x)$, where $p(x)=1+x-x^2$, $q(x)=1-x^3-x^6$ and
$r(x)=1+x^3-x^6=-x^6 q(x^{-1})$, we deduce immediately that 
$N_U(x)=N_V(x)$. 

This counter-example can be easily generalized. Let us define
{\em binary polynomials} as polynomials associated to binary
sequences. If $f(x)$ is a polynomial of degree $d$ with
$f(0)\ne0$, we define its {\em dual} polynomial $f^*$ by 
$f^*(x)=x^d f(x^{-1})$. Then for any positive integer $k$
we have $f^*(x^k)=f(x^k)^*$, i.e., $f^*(x^k)=g^*(x)$ where
$g^*$ is the dual of the polynomial $f(x^k)$.
In general we can start with any number of binary sequences, but 
here we take only three of them: $A;\,B;\,C$ of lengths $m,n,k$
respectively. From the associated binary polynomials
$A(x),B(x),C(x)$ we can form several binary polynomials of
degree $mnk-1$. The basic one is $A(x)B(x^m)C(x^{mn})$.
The other are obtained from this one by replacing one or
more of the three factors by their duals.
It is immediate that the binary sequences corresponding
to these binary polynomials all have the same norm.
In general many of these sequences will not be equivalent.
However note that if we replace all three factors
with their duals, we will obtain a binary sequence equivalent
to the basic one.

\section{Acknowledgments}

The author has pleasure to thank the three anonymous referees for their valuable comments. The author is grateful to NSERC for the continuing support of his research. This work was made possible by the facilities 
of the Shared Hierarchical Academic Research Computing Network 
(SHARCNET:www.sharcnet.ca) and Compute/Calcul Canada.

\newpage

\begin{center}
\begin{tabular}{rlll}
\multicolumn{4}{c}{Table 2: Equivalence classes of $BS(n+1,n)$,
$n\le8$} \\ \hline 
\multicolumn{1}{c}{} & \multicolumn{1}{c}{$ABCD$} & 
\multicolumn{1}{c}{$N_A$ \& $N_B$} &
\multicolumn{1}{c}{$N_C$ \& $N_D$} \\ \hline
\multicolumn{4}{c}{ $n=1$ }\\
1 & 0 &  $2,1$ & $1$ \\
 & 0 & $2,-1$ & $1$ \\
\multicolumn{4}{c}{ $n=2$ }\\
1 & 03 &  $3,-2,1$ & $2,1$ \\ 
 & 1 & $3,0,-1$ & $2,1$ \\
\multicolumn{4}{c}{ $n=3$ }\\
1 & 06 &  $4,-1,0,1$ & $3,2,1$ \\
 & 11 & $4,1,-2,-1$ & $3,-2,1$ \\
\multicolumn{4}{c}{ $n=4$ }\\
1 & 060 &  $5,0,1,0,1$ & $4,-1,0,1$ \\
 & 16 & $5,2,-1,-2,-1$ & $4,-1,0,1$ \\
2 & 082 &  $5,0,-1,-2,1$ & $4,3,2,1$ \\
 & 12 & $5,-2,1,0,-1$ & $4,-1,-2,1$ \\
3 & 083 &  $5,0,-1,-2,1$ & $4,-1,0,1$ \\
 & 16 & $5,2,1,0,-1$ & $4,-1,0,1$ \\
\multicolumn{4}{c}{ $n=5$ }\\
1 & 016 &  $6,1,0,1,2,1$ & $5,2,-1,-2,-1$ \\
 & 640 & $6,-1,2,-1,0,-1$ & $5,-2,-1,2,-1$ \\
2 & 017 &  $6,1,-4,-1,2,1$ & $5,-2,1,0,-1$ \\
 & 613 & $6,3,2,1,0,-1$ & $5,-2,1,0,-1$ \\
3 & 064 &  $6,1,-2,-1,0,1$ & $5,0,1,0,1$ \\
 & 160 & $6,-1,0,1,-2,-1$ & $5,0,1,0,1$ \\
4 & 065 &  $6,-3,2,-1,0,1$ & $5,0,-1,2,1$  \\
 & 113 & $6,3,0,-3,-2,-1$ & $5,0,-1,2,1$ \\
\multicolumn{4}{c}{ $n=6$ }\\
1 & 0612 &  $7,-2,3,-2,1,0,1$ & $6,1,-4,-1,2,1$ \\
 & 127 & $7,4,1,0,-1,-2,-1$ & $6,-3,0,3,-2,1$ \\
2 & 0820 &  $7,-2,1,0,3,-2,1$ & $6,1,0,-1,-2,1$ \\
 & 188 & $7,0,-1,2,1,0,-1$ & $6,1,0,-1,-2,1$ \\
3 & 0861 &  $7,-2,-3,2,1,-2,1$ & $6,1,2,1,0,1$ \\
 & 162 & $7,0,3,0,-1,0,-1$ & $6,1,-2,-3,0,1$ \\
4 & 0872 & $7,2,1,0,-1,-2,1$ & $6,1,0,1,2,1$ \\
 & 126 & $7,0,-1,-2,1,0,-1$ & $6,-3,0,1,-2,1$ \\ 
5 & 0882 &  $7,2,1,0,-1,-2,1$ & $6,1,-2,-1,0,1$ \\
 & 164 & $7,0,-1,2,1,0,-1$ & $6,-3,2,-1,0,1$ \\
\multicolumn{4}{c}{ $n=7$ }\\
1 & 0165 & $8,-1,2,-1,0,1,2,1$ & $7,0,-1,2,1,0,-1$ \\ 
 & 6123 & $8,1,0,1,-2,-1,0,-1$ & $7,0,-1,-2,1,0,-1$ \\ 
2 & 0165 & $8,-1,2,-1,0,1,2,1$ & $7,0,3,0,-1,0,-1$ \\ 
 & 6141 & $8,1,0,1,-2,-1,0,-1$ & $7,0,-5,0,3,0,-1$ \\ 
\hline
\end{tabular} \\
\end{center}

\begin{center}
\begin{tabular}{rlll}
\multicolumn{4}{c}{Table 2 (continued)} \\ \hline 
\multicolumn{1}{c}{} & \multicolumn{1}{c}{$ABCD$} & 
\multicolumn{1}{c}{$N_A$ \& $N_B$} &
\multicolumn{1}{c}{$N_C$ \& $N_D$} \\ \hline
\multicolumn{4}{c}{ $n=7$ }\\
3 & 0166 & $8,3,-2,-1,0,1,2,1$ & $7,0,-1,2,1,0,-1$ \\ 
 & 6122 & $8,1,0,1,-2,-1,0,-1$ & $7,-4,3,-2,1,0,-1$ \\ 
4 & 0173 & $8,-1,-2,1,-2,-1,2,1$ & $7,0,3,0,-1,0,-1$ \\ 
 & 6161 & $8,1,0,-1,4,1,0,-1$ & $7,0,-1,0,-1,0,-1$ \\ 
5 & 0173 & $8,-1,-2,1,-2,-1,2,1$ & $7,4,1,0,-1,-2,-1$ \\ 
 & 6411 & $8,1,0,-1,4,1,0,-1$ & $7,-4,1,0,-1,2,-1$ \\ 
6 & 0183 & $8,-1,-2,1,-2,-1,2,1$ & $7,4,3,2,1,0,-1$ \\ 
 & 6121 & $8,-3,0,-1,0,1,0,-1$ & $7,0,-1,-2,1,0,-1$ \\
7 & 0613 & $8,-1,0,3,0,1,0,1$ & $7,-2,3,-2,1,0,1$ \\ 
 & 1623 & $8,1,-2,1,2,-1,-2,-1$ & $7,2,-1,-2,-3,0,1$ \\ 
8 & 0614 & $8,-1,4,-1,0,1,0,1$ & $7,2,-1,0,-1,0,1$ \\ 
 & 1641 & $8,1,-2,1,2,-1,-2,-1$ & $7,-2,-1,0,-1,0,1$ \\ 
9 & 0615 & $8,-1,0,3,0,1,0,1$ & $7,2,-3,-2,1,2,1$ \\ 
 & 1263 & $8,1,2,1,-2,-1,-2,-1$ & $7,-2,1,-2,1,-2,1$ \\ 
10 & 0615 & $8,-1,0,3,0,1,0,1$ & $7,2,-3,-4,-1,2,1$ \\ 
 & 1272 & $8,1,2,1,-2,-1,-2,-1$ & $7,-2,1,0,3,-2,1$ \\ 
11 & 0616 & $8,-1,4,-1,0,1,0,1$ & $7,2,-3,-2,1,2,1$ \\ 
 & 1262 & $8,1,2,1,-2,-1,-2,-1$ & $7,-2,-3,2,1,-2,1$ \\ 
12 & 0618 & $8,-1,0,-1,-2,1,0,1$ & $7,2,1,2,1,2,1$ \\ 
 & 1261 & $8,1,-2,1,0,-1,-2,-1$ & $7,-2,1,-2,1,-2,1$ \\ 
13 & 0635 & $8,-1,-4,1,2,-1,0,1$ & $7,2,3,2,1,0,1$ \\ 
 & 1621 & $8,-3,2,-1,0,1,-2,-1$ & $7,2,-1,-2,-3,0,1$ \\ 
14 & 0638 & $8,-1,0,-3,0,-1,0,1$ & $7,2,3,2,1,0,1$ \\ 
 & 1620 & $8,1,-2,-1,2,1,-2,-1$ & $7,-2,-1,2,-3,0,1$ \\ 
15 & 0641 & $8,3,0,1,0,-1,0,1$ & $7,-2,3,-2,1,0,1$ \\ 
 & 1622 & $8,1,-2,-1,2,1,-2,-1$ & $7,-2,-1,2,-3,0,1$ \\ 
16 & 0646 & $8,3,0,-3,-2,-1,0,1$ & $7,2,1,0,3,2,1$ \\ 
 & 1222 & $8,-3,2,-1,0,1,-2,-1$ & $7,-2,-3,4,-1,-2,1$ \\ 
17 & 0646 & $8,3,0,-3,-2,-1,0,1$ & $7,2,1,2,1,2,1$ \\ 
 & 1260 & $8,-3,2,-1,0,1,-2,-1$ & $7,-2,-3,2,1,-2,1$ \\
\multicolumn{4}{c}{ $n=8$ }\\
1 & 06113 & $9,0,1,4,-1,2,1,0,1$ & $8,-1,0,-3,0,-1,0,1$ \\
 & 1638 & $9,2,-1,2,1,0,-1,-2,-1$ & $8,-1,0,-3,0,-1,0,1$ \\
2 & 06122 & $9,0,1,4,-1,2,1,0,1$ & $8,3,0,-3,-2,-1,0,1$ \\
 & 1644 & $9,-2,3,-2,1,0,-1,-2,-1$ & $8,-1,-4,1,2,-1,0,1$ \\
3 & 06141 & $9,0,5,0,1,0,1,0,1$ & $8,-1,0,1,-2,-1,0,1$ \\
 & 1663 & $9,2,-5,-2,3,2,-1,-2,-1$ & $8,-1,0,1,-2,-1,0,1$ \\
4 & 06142 & $9,0,1,0,-3,0,1,0,1$ & $8,-1,4,-1,0,1,0,1$ \\
 & 1624 & $9,2,-1,-2,3,2,-1,-2,-1$ & $8,-1,-4,3,0,-3,0,1$ \\
\hline
\end{tabular} \\
\end{center}

\begin{center}
\begin{tabular}{rlll}
\multicolumn{4}{c}{Table 2 (continued)} \\ \hline 
\multicolumn{1}{c}{} & \multicolumn{1}{c}{$ABCD$} & 
\multicolumn{1}{c}{$N_A$ \& $N_B$} &
\multicolumn{1}{c}{$N_C$ \& $N_D$} \\ \hline
\multicolumn{4}{c}{ $n=8$ }\\
5 & 06151 & $9,0,1,0,5,0,1,0,1$ & $8,-1,0,-1,-2,1,0,1$ \\
 & 1618 & $9,2,-1,2,-1,-2,-1,-2,-1$ & $8,-1,0,-1,-2,1,0,1$ \\
6 & 06152 & $9,0,-3,0,1,0,1,0,1$ & $8,3,-2,-1,0,1,2,1$ \\
 & 1264 & $9,2,3,2,-1,-2,-1,-2,-1$ & $8,-5,2,-1,0,1,-2,1$ \\
7 & 06183 & $9,0,1,-4,-1,-2,1,0,1$ & $8,-1,-2,5,0,-1,2,1$ \\
 & 1271 & $9,2,-1,-2,1,0,-1,-2,-1$ & $8,-1,2,1,0,3,-2,1$ \\
8 & 06183 & $9,0,1,-4,-1,-2,1,0,1$ & $8,-1,0,3,0,1,0,1$ \\
 & 1613 & $9,2,-1,-2,1,0,-1,-2,-1$ & $8,-1,0,3,0,1,0,1$ \\
9 & 06310 & $9,0,1,2,1,4,-1,0,1$ & $8,-1,0-1,0,-3,0,1$ \\
 & 1686 & $9,2,-1,0,-1,2,1,-2,-1$ & $8,-1,0-1,0,-3,0,1$ \\
10 & 06380 & $9,-4,1,-2,1,0,-1,0,1$ & $8,3,0,1,0,-1,0,1$ \\
& 1661 & $9,-2,-1,0,-1,2,1,-2,-1$ & $8,3,0,1,0,-1,0,1$ \\
11 & 06382 & $9,0,1,-2,-3,0,-1,0,1$ & $8,3,0,1,0,-1,0,1$ \\
& 1641 & $9,-2,-1,0,-1,2,1,-2,-1$ & $8,-1,0,1,4,-1,0,1$ \\
12 & 06412 & $9,0,1,2,-3,0,-1,0,1$ & $8,-1,0,1,4,-1,0,1$ \\
& 1632 & $9,2,-1,0,-1,2,1,-2,-1$ & $8,-1,0,-3,0,-1,0,1$ \\
13 & 06580 & $9,-4,1,-2,1,0,-1,0,1$ & $8,3,-2,-3,0,3,2,1$ \\
& 1127 & $9,2,3,0,-1,-2,-3,-2,-1$ & $8,-1,-2,5,0,-1,2,1$ \\
14 & 06633 & $9,0,-3,2,-1,-2,-1,0,1$ & $8,-1,2,-1,0,1,2,1$ \\
& 1163 & $9,2,-1,0,1,0,-3,-2,-1$ & $8,-1,2,-1,0,1,2,1$ \\
15 & 06852 & $9,0,-3,0,1,0,-3,0,1$ & $8,-1,2,-1,0,1,2,1$ \\
 & 1163 & $9,2,-1,2,-1,-2,-1,-2,-1$ & $8,-1,2,-1,0,1,2,1$ \\
16 & 06860 & $9,0,-3,0,1,0,-3,0,1$ & $8,-1,2,-1,0,1,2,1$ \\
 & 1163 & $9,2,-1,2,-1,-2,-1,-2,-1$ & $8,-1,2,-1,0,1,2,1$ \\
17 & 08110 & $9,0,3,2,1,0,3,-2,1$ & $8,-1,-2,-1,0,1,-2,1$ \\
 & 1866 & $9,2,1,0,-1,-2,1,0,-1$ & $8,-1,-2,-1,0,1,-2,1$ \\
18 & 08350 & $9,0,-1,-4,1,0,1,-2,1$ & $8,-1,2,1,0,3,-2,1$ \\
 & 1822 & $9,-2,-3,2,-1,-2,3,0,-1$ & $8,3,2,1,0,-1,-2,1$ \\
19 & 08383 & $9,0,3,0,-1,-2,1,-2,1$ & $8,-1,-2,-1,0,1,-2,1$ \\
 & 1866 & $9,2,1,2,1,0,3,0,-1$ & $8,-1,-2,-1,0,1,-2,1$ \\
20 & 08630 & $9,-4,-1,0,1,0,1,-2,1$ & $8,3,0,1,0,-1,0,1$ \\
 & 1661 & $9,-2,1,-2,-1,2,-1,0,-1$ & $8,3,0,1,0,-1,0,1$ \\
21 & 08640 & $9,0,-1,-4,1,0,1,-2,1$ & $8,-1,-2,5,0,-1,2,1$ \\
 & 1282 & $9,-2,1,-2,-1,2,-1,0,-1$ & $8,3,2,1,0,-1,-2,1$ \\
22 & 08642 & $9,0,-1,0,-3,0,1,-2,1$ & $8,3,0,1,0,-1,0,1$ \\
 & 1641 & $9,-2,1,-2,-1,2,-1,0,-1$ & $8,-1,0,1,4,-1,0,1$ \\
\hline
\end{tabular} \\
\end{center}

\begin{center}
\begin{tabular}{rlll}
\multicolumn{4}{c}{Table 2 (continued)} \\ \hline 
\multicolumn{1}{c}{} & \multicolumn{1}{c}{$ABCD$} & 
\multicolumn{1}{c}{$N_A$ \& $N_B$} &
\multicolumn{1}{c}{$N_C$ \& $N_D$} \\ \hline
\multicolumn{4}{c}{ $n=8$ }\\
23 & 08660 & $9,0,-1,-4,1,0,1,-2,1$ & $8,-1,-2,5,0,-1,2,1$ \\
 & 1271 & $9,2,1,-2,-1,-2,-1,0,-1$ & $8,-1,2,1,0,3,-2,1$ \\
24 & 08660 & $9,0,-1,-4,1,0,1,-2,1$ & $8,-1,0,3,0,1,0,1$ \\
 & 1613 & $9,2,1,-2,-1,-2,-1,0,-1$ & $8,-1,0,3,0,1,0,1$ \\
25 & 08833 & $9,0,-1,2,-1,2,-1,-2,1$ & $8,-1,0,-1,0,-3,0,1$ \\
 & 1686 & $9,2,1,0,1,4,1,0,-1$ & $8,-1,0,-1,0,-3,0,1$ \\
26 & 08862 & $9,0,-1,2,-1,2,-1,-2,1$ & $8,-1,4,-1,0,1,0,1$ \\
 & 1626 & $9,2,-3,0,1,0,1,0,-1$ & $8,-1,0,-1,0,-3,0,1$ \\
27 & 08863 & $9,0,-1,2,-1,2,-1,-2,1$ & $8,-1,0,-3,0,-1,0,1$ \\
 & 1638 & $9,2,1,4,1,0,1,0,-1$ & $8,-1,0,-3,0,-1,0,1$ \\
\hline \\
\end{tabular} \\
\end{center}

\begin{center}
\begin{tabular}{rl|rl|rl|rl}
\multicolumn{8}{c}{Table 3: Equivalence classes of $BS(10,9)$} \\ \hline 
\multicolumn{1}{c}{} & \multicolumn{1}{c}{$AB$ \quad $CD$} & 
\multicolumn{1}{c}{} & \multicolumn{1}{c}{$AB$ \quad $CD$} & 
\multicolumn{1}{c}{} & \multicolumn{1}{c}{$AB$ \quad $CD$} & 
\multicolumn{1}{c}{} & \multicolumn{1}{c}{$AB$ \quad $CD$} \\ \hline
1 & 01235 66450 & 2 & 01324 66181 & 3 & 01618 64150 & 4 & 01624 64183 \\
5 & 01627 64130 & 6 & 01633 64140 & 7 & 01642 64560 & 8 & 01652 61453 \\
9 & 01652 64313 & 10 & 01654 61163 & 11 & 01655 61180 & 12 & 01672 61281 \\
13 & 01675 61430 & 14 & 01682 61180 & 15 & 01684 61122 & 16 & 01734 64160 \\
17 & 01735 61640 & 18 & 01764 61821 & 19 & 01765 61281 & 20 & 01767 61831 \\
21 & 01783 61411 & 22 & 01867 61311 & 23 & 06124 16282 & 24 & 06136 16640 \\
25 & 06147 16450 & 26 & 06152 12763 & 27 & 06164 16133 & 28 & 06172 12681 \\
29 & 06175 12670 & 30 & 06175 16143 & 31 & 06187 16131 & 32 & 06351 16460 \\
33 & 06382 16460 & 34 & 06388 16340 & 35 & 06412 16273 & 36 & 06412 16381 \\
37 & 06413 16460 & 38 & 06451 16163 & 39 & 06458 12612 & 40 & 06481 12623 \\
41 & 06481 16161 & 42 & 06581 11622 & 43 & 06583 11631 & 44 & 06875 11622 \\
\hline
\end{tabular} \\
\end{center}

\newpage

\begin{center}
\begin{tabular}{rl|rl|rl}
\multicolumn{6}{c}{Table 4: Equivalence classes of $BS(11,10)$} \\ \hline 
\multicolumn{1}{c}{} & \multicolumn{1}{c}{$AB$ \quad $CD$} & 
\multicolumn{1}{c}{} & \multicolumn{1}{c}{$AB$ \quad $CD$} & 
\multicolumn{1}{c}{} & \multicolumn{1}{c}{$AB$ \quad $CD$} \\ \hline
1 & 061173 16456 & 2 & 061253 16246 & 3 & 061350 16645 \\
4 & 061450 16267 & 5 & 061450 16443 & 6 & 061460 16434 \\
7 & 061463 12826 & 8 & 061463 16271 & 9 & 061553 12716 \\
10 & 061563 16134 & 11 & 061582 12631 & 12 & 061633 12671 \\
13 & 061740 12684 & 14 & 061870 12286 & 15 & 061870 16144 \\
16 & 063140 16862 & 17 & 063413 16822 & 18 & 063510 16382 \\
19 & 063513 16441 & 20 & 063550 16414 & 21 & 063583 16341 \\
22 & 063810 16616 & 23 & 063821 16445 & 24 & 063833 16613 \\
25 & 063840 16262 & 26 & 063842 16242 & 27 & 063870 16322 \\
28 & 063873 16217 & 29 & 063881 16342 & 30 & 064122 16277 \\
31 & 064130 16465 & 32 & 064141 16423 & 33 & 064141 16452 \\
34 & 064170 16344 & 35 & 064313 16282 & 36 & 064413 12826 \\
37 & 064413 16271 & 38 & 064480 16213 & 39 & 064480 16321 \\
40 & 064510 12638 & 41 & 064510 12676 & 42 & 064870 12616 \\
43 & 065843 11276 & 44 & 068110 11863 & 45 & 068383 11863 \\
46 & 068560 11263 & 47 & 068571 11632 & 48 & 068580 11627 \\
49 & 068580 11643 & 50 & 068611 11634 & 51 & 068632 11632 \\
52 & 068641 11634 & 53 & 068752 11276 & 54 & 068771 11645 \\
55 & 082661 18642 & 56 & 083110 18863 & 57 & 083383 18863 \\
58 & 083510 18226 & 59 & 083521 18642 & 60 & 083850 18622 \\
61 & 086231 16248 & 62 & 086243 16277 & 63 & 086263 16332 \\
64 & 086310 16613 & 65 & 086333 16616 & 66 & 086343 16228 \\
67 & 086421 16427 & 68 & 086432 16242 & 69 & 086463 16217 \\
70 & 086473 16217 & 71 & 086483 16344 & 72 & 086532 16142 \\
73 & 086640 12631 & 74 & 086643 16134 & 75 & 086740 12642 \\
76 & 086840 12682 & 77 & 086860 12671 & 78 & 086870 12671 \\
79 & 087110 12863 & 80 & 087110 16273 & 81 & 087120 16461 \\
82 & 087130 16461 & 83 & 087131 16262 & 84 & 087221 16284 \\
85 & 087323 16282 & 86 & 087343 16282 & 87 & 087361 16422 \\
88 & 087372 12864 & 89 & 087383 16382 & 90 & 087663 16146 \\
91 & 087683 12637 & 92 & 087732 12684 & 93 & 088651 16264 \\
94 & 088651 16424 & 95 & 088673 16434 & 96 & 088762 16246 \\
97 & 088771 16264 & 98 & 088771 16424 & & \\
\hline
\end{tabular} \\
\end{center}

\begin{center}
\begin{tabular}{rl|rl|rl}
\multicolumn{6}{c}{Table 5: Equivalence classes of $BS(12,11)$} \\ \hline 
\multicolumn{1}{c}{} & \multicolumn{1}{c}{$AB$ \quad $CD$} & 
\multicolumn{1}{c}{} & \multicolumn{1}{c}{$AB$ \quad $CD$} & 
\multicolumn{1}{c}{} & \multicolumn{1}{c}{$AB$ \quad $CD$} \\ \hline
1 & 011823 661422 & 2 & 012356 661422 & 3 & 013682 663120 \\
4 & 013753 663120 & 5 & 016426 641272 & 6 & 016445 616230 \\
7 & 016472 645121 & 8 & 016525 614232 & 9 & 016525 643511 \\
10 & 016535 612770 & 11 & 016542 612463 & 12 & 016542 612843 \\
13 & 016542 614242 & 14 & 016546 612440 & 15 & 016572 614320 \\
16 & 016634 614123 & 17 & 016634 614231 & 18 & 016643 612271 \\
19 & 016653 612313 & 20 & 016724 643121 & 21 & 016727 612440 \\
22 & 016756 611280 & 23 & 016774 612363 & 24 & 016817 614320 \\
25 & 017262 641243 & 26 & 017356 616123 & 27 & 017374 616242 \\
28 & 017375 641243 & 29 & 017632 614620 & 30 & 017646 612640 \\
31 & 017664 614520 & 32 & 017674 612462 & 33 & 017674 614272 \\
34 & 018265 612772 & 35 & 018342 614620 & 36 & 018382 612770 \\
37 & 018767 613123 & 38 & 018767 613230 & 39 & 061186 164231 \\
40 & 061246 166323 & 41 & 061256 164341 & 42 & 061264 162273 \\
43 & 061264 162433 & 44 & 061284 164231 & 45 & 061462 162452 \\
46 & 061462 164322 & 47 & 061463 162262 & 48 & 061472 128681 \\
49 & 061473 164251 & 50 & 061476 162431 & 51 & 061476 164320 \\
52 & 061547 127621 & 53 & 061575 126232 & 54 & 061618 126232 \\
55 & 061624 126332 & 56 & 061644 126341 & 57 & 061764 126242 \\
58 & 061774 126343 & 59 & 063144 168281 & 60 & 063412 168481 \\
61 & 063512 162642 & 62 & 063515 162441 & 63 & 063515 164321 \\
64 & 063541 162660 & 65 & 063541 164420 & 66 & 063551 163241 \\
67 & 063811 164261 & 68 & 063824 162660 & 69 & 063824 164420 \\
70 & 063825 162772 & 71 & 063828 162273 & 72 & 063828 162433 \\
73 & 063842 164261 & 74 & 063858 163320 & 75 & 063877 128160 \\
76 & 063882 163423 & 77 & 063884 163421 & 78 & 064143 164251 \\
79 & 064146 162431 & 80 & 064146 164320 & 81 & 064615 122632 \\
82 & 064826 126262 & 83 & 064838 126451 & 84 & 064842 126451 \\
\hline
\end{tabular} \\
\end{center}

\begin{center}
\begin{tabular}{rl|rl|rl}
\multicolumn{6}{c}{Table 6: Equivalence classes of $BS(13,12)$} \\ \hline 
\multicolumn{1}{c}{} & \multicolumn{1}{c}{$AB$ \quad $CD$} & 
\multicolumn{1}{c}{} & \multicolumn{1}{c}{$AB$ \quad $CD$} & 
\multicolumn{1}{c}{} & \multicolumn{1}{c}{$AB$ \quad $CD$} \\ \hline
1 & 0611863 164521 & 2 & 0611871 166143 & 3 & 0612360 164844 \\
4 & 0612573 164215 & 5 & 0612760 128287 & 6 & 0612760 162443 \\
7 & 0612870 162424 & 8 & 0614230 164864 & 9 & 0614533 164226 \\
10 & 0614670 162327 & 11 & 0614671 128266 & 12 & 0614830 162624 \\
13 & 0615230 126876 & 14 & 0615230 161883 & 15 & 0615643 161234 \\
16 & 0615733 126452 & 17 & 0616433 126452 & 18 & 0616533 126143 \\
19 & 0616733 126143 & 20 & 0617212 126847 & 21 & 0617220 126876 \\
22 & 0617220 161883 & 23 & 0617623 126413 & 24 & 0617651 122818 \\
25 & 0617811 126428 & 26 & 0617820 161624 & 27 & 0618833 126413 \\
28 & 0631412 168644 & 29 & 0634170 168382 & 30 & 0635120 164558 \\
31 & 0635152 164341 & 32 & 0635340 164621 & 33 & 0635440 164242 \\
34 & 0635483 162246 & 35 & 0635483 162432 & 36 & 0635513 164132 \\
37 & 0635550 163214 & 38 & 0635810 163422 & 39 & 0635810 164142 \\
40 & 0635872 162134 & 41 & 0638112 163822 & 42 & 0638121 164423 \\
43 & 0638212 164242 & 44 & 0638220 164278 & 45 & 0638241 164423 \\
46 & 0638781 162167 & 47 & 0641282 162424 & 48 & 0641363 164522 \\
49 & 0641370 164265 & 50 & 0641470 162732 & 51 & 0641471 128642 \\
52 & 0641471 164217 & 53 & 0641481 164215 & 54 & 0643513 164242 \\
55 & 0643822 164522 & 56 & 0643880 162424 & 57 & 0644833 163214 \\
58 & 0646430 126143 & 59 & 0648112 126283 & 60 & 0648212 161624 \\
61 & 0648213 161644 & 62 & 0648363 126461 & 63 & 0648382 126716 \\
64 & 0648433 126452 & 65 & 0658112 116273 & 66 & 0658363 116342 \\
67 & 0658463 112645 & 68 & 0661363 116245 & 69 & 0661453 112634 \\
70 & 0663640 116342 & 71 & 0663853 116324 & 72 & 0663880 116245 \\
73 & 0664360 116245 & 74 & 0685733 116245 & 75 & 0685860 112645 \\
76 & 0685871 112766 & 77 & 0686130 116245 & 78 & 0686230 116636 \\
79 & 0686240 116273 & 80 & 0686433 116245 & 81 & 0687623 112645 \\
82 & 0688613 116245 & 83 & 0688671 112764 & 84 & 0811283 182788 \\
85 & 0812661 182668 & 86 & 0812883 186247 & 87 & 0816621 181128 \\
88 & 0817883 181128 & 89 & 0826620 186247 & 90 & 0826782 182266 \\
91 & 0826783 182641 & 92 & 0826851 186444 & 93 & 0835382 182266 \\
94 & 0836121 186621 & 95 & 0837383 182278 & 96 & 0837383 182771 \\
97 & 0838521 186627 & 98 & 0838652 182777 & 99 & 0838673 182668 \\
100 & 0838751 182777 & 101 & 0838871 186644 & 102 & 0862230 164278 \\
103 & 0862270 164413 & 104 & 0862312 164226 & 105 & 0862322 166128 \\
106 & 0862383 162748 & 107 & 0862441 164265 & 108 & 0862483 128674 \\
109 & 0862650 163341 & 110 & 0862651 162138 & 111 & 0863220 164287 \\
112 & 0863353 162642 & 113 & 0863353 164612 & 114 & 0863482 166144 \\
115 & 0864121 164324 & 116 & 0864311 164226 & 117 & 0864343 164226 \\
118 & 0864382 164413 & 119 & 0864463 162451 & 120 & 0864781 162167 \\
\hline
\end{tabular} \\
\end{center}

\begin{center}
\begin{tabular}{rl|rl|rl}
\multicolumn{6}{c}{Table 6: (continued)} \\ \hline 
\multicolumn{1}{c}{} & \multicolumn{1}{c}{$AB$ \quad $CD$} & 
\multicolumn{1}{c}{} & \multicolumn{1}{c}{$AB$ \quad $CD$} & 
\multicolumn{1}{c}{} & \multicolumn{1}{c}{$AB$ \quad $CD$} \\ \hline
121 & 0865261 126237 & 122 & 0865310 126452 & 123 & 0865343 161642 \\
124 & 0865382 126238 & 125 & 0865382 161267 & 126 & 0866310 126413 \\
127 & 0866443 126314 & 128 & 0867740 126174 & 129 & 0868761 126238 \\
130 & 0868761 161267 & 131 & 0871130 162747 & 132 & 0871332 162624 \\
133 & 0872231 162838 & 134 & 0872833 164432 & 135 & 0873111 166361 \\
136 & 0873470 162842 & 137 & 0873470 166413 & 138 & 0873481 162862 \\
139 & 0873581 164413 & 140 & 0873583 162445 & 141 & 0873670 162248 \\
142 & 0873681 162268 & 143 & 0873711 164226 & 144 & 0873750 164413 \\
145 & 0876221 127763 & 146 & 0876510 126342 & 147 & 0876510 126451 \\
148 & 0876570 122864 & 149 & 0876570 161247 & 150 & 0876581 126238 \\
151 & 0876581 161267 & 152 & 0876612 126341 & 153 & 0876663 126314 \\
154 & 0876750 126238 & 155 & 0876750 161267 & 156 & 0877382 126876 \\
157 & 0877382 161883 & 158 & 0877871 126748 & 159 & 0878631 127778 \\
160 & 0878663 127647 & 161 & 0878861 161886 & 162 & 0878870 161884 \\
163 & 0881211 168382 & 164 & 0882762 168242 & 165 & 0883571 168242 \\
166 & 0883671 168422 & 167 & 0886261 162874 & 168 & 0886280 162867 \\
169 & 0886471 162842 & 170 & 0886471 166413 & 171 & 0886560 164215 \\
172 & 0886613 164521 & 173 & 0886671 162741 & 174 & 0886760 164215 \\
175 & 0887780 162748 &     &                &     & \\
\hline
\end{tabular} \\
\end{center}

\begin{center}
\begin{tabular}{rl|rl|rl}
\multicolumn{6}{c}{Table 7: Equivalence classes of $BS(14,13)$} \\ \hline 
\multicolumn{1}{c}{} & \multicolumn{1}{c}{$AB$ \quad $CD$} & 
\multicolumn{1}{c}{} & \multicolumn{1}{c}{$AB$ \quad $CD$} & 
\multicolumn{1}{c}{} & \multicolumn{1}{c}{$AB$ \quad $CD$} \\ \hline
1& 0116455 6616380& 2& 0116536 6645183& 3& 0116546 6645153 \\
4& 0116734 6641822& 5& 0117653 6618273& 6& 0117653 6618422 \\
7& 0117663 6618450& 8& 0118176 6618441& 9& 0118323 6618222 \\
10& 0118324 6618241& 11& 0118327 6612712& 12& 0118327 6614122 \\
13& 0118345 6614212& 14& 0123628 6645153& 15& 0123644 6641450 \\
16& 0123672 6614413& 17& 0123827 6611811& 18& 0131554 6618272 \\
19& 0131657 6618273& 20& 0131657 6618422& 21& 0131676 6612763 \\
22& 0131735 6641851& 23& 0131743 6644513& 24& 0131745 6645560 \\
25& 0131842 6618222& 26& 0132157 6641822& 27& 0132284 6614580 \\
28& 0132354 6614153& 29& 0132427 6641810& 30& 0132463 6618843 \\
31& 0133154 6612743& 32& 0133414 6618451& 33& 0133425 6611813 \\
34& 0134174 6641273& 35& 0134174 6641853& 36& 0134234 6641513 \\
37& 0134246 6645113& 38& 0134273 6612412& 39& 0134274 6612280 \\
40& 0134314 6641851& 41& 0134416 6641822& 42& 0136167 6636440 \\
43& 0136764 6634111& 44& 0137536 6631223& 45& 0138324 6631411 \\
46& 0161327 6441863& 47& 0161533 6418643& 48& 0161633 6414853 \\
49& 0161655 6456330& 50& 0161742 6418643& 51& 0161755 6412743 \\
52& 0161762 6418430& 53& 0162173 6162842& 54& 0162283 6415680 \\
55& 0162328 6418650& 56& 0162556 6451340& 57& 0162582 6451343 \\
58& 0162617 6451282& 59& 0162766 6413151& 60& 0163157 6168433 \\
61& 0163174 6168441& 62& 0163255 6164143& 63& 0163287 6162131 \\
64& 0163325 6168241& 65& 0163352 6162760& 66& 0163372 6414161 \\
67& 0163375 6412612& 68& 0163417 6414653& 69& 0163462 6412681 \\
70& 0163562 6451883& 71& 0163822 6455630& 72& 0163844 6451220 \\
73& 0164133 6456613& 74& 0164143 6168242& 75& 0164156 6456550 \\
76& 0164234 6164450& 77& 0164314 6162862& 78& 0164317 6162871 \\
79& 0164327 6412623& 80& 0164367 6164311& 81& 0164471 6418541 \\
82& 0164481 6411811& 83& 0164624 6451133& 84& 0165173 6142832 \\
85& 0165174 6142582& 86& 0165243 6141643& 87& 0165274 6122463 \\
88& 0165327 6142412& 89& 0165347 6141622& 90& 0165367 6142312 \\
91& 0165413 6144831& 92& 0165428 6142152& 93& 0165523 6143513 \\
94& 0165716 6143581& 95& 0165816 6114242& 96& 0165826 6114531 \\
97& 0165826 6123142& 98& 0166137 6128182& 99& 0166153 6127433 \\
100& 0166173 6128371& 101& 0166317 6127481& 102& 0166324 6124841 \\
103& 0166413 6127832& 104& 0166423 6141830& 105& 0166543 6112471 \\
106& 0167125 6142682& 107& 0167134 6148463& 108& 0167156 6434350 \\
109& 0167162 6127643& 110& 0167162 6142763& 111& 0167238 6128222 \\
112& 0167286 6122162& 113& 0167345 6144161& 114& 0167356 6141461 \\
115& 0167356 6183881& 116& 0167365 6144311& 117& 0167385 6116270 \\
118& 0167416 6144581& 119& 0167426 6142280& 120& 0167455 6431311 \\
\hline
\end{tabular} \\
\end{center}

\begin{center}
\begin{tabular}{rl|rl|rl}
\multicolumn{6}{c}{Table 7: (continued)} \\ \hline 
\multicolumn{1}{c}{} & \multicolumn{1}{c}{$AB$ \quad $CD$} & 
\multicolumn{1}{c}{} & \multicolumn{1}{c}{$AB$ \quad $CD$} & 
\multicolumn{1}{c}{} & \multicolumn{1}{c}{$AB$ \quad $CD$} \\ \hline
121& 0167457 6435150& 122& 0167465 6141272& 123& 0167466 6124512 \\
124& 0167583 6112740& 125& 0167584 6123240& 126& 0167817 6112763 \\
127& 0168171 6118282& 128& 0168171 6143841& 129& 0168276 6114521 \\
130& 0168286 6114531& 131& 0168286 6123142& 132& 0168425 6143411 \\
133& 0168465 6112422& 134& 0168476 6123420& 135& 0168486 6114530 \\
136& 0171655 6416273& 137& 0171656 6416851& 138& 0171831 6416481 \\
139& 0172647 6411422& 140& 0172656 6415223& 141& 0173362 6416122 \\
142& 0173413 6441863& 143& 0173474 6445160& 144& 0173521 6162781 \\
145& 0173523 6414151& 146& 0173554 6161242& 147& 0173612 6164880 \\
148& 0173612 6414863& 149& 0173658 6412511& 150& 0173744 6162760 \\
151& 0173843 6412141& 152& 0173843 6412710& 153& 0176164 6182672 \\
154& 0176421 6146273& 155& 0176421 6184643& 156& 0176424 6146511 \\
157& 0176515 6434350& 158& 0176526 6431511& 159& 0176541 6124682 \\
160& 0176542 6142280& 161& 0176554 6141272& 162& 0176557 6145313 \\
163& 0176583 6145310& 164& 0176636 6142172& 165& 0176641 6142780 \\
166& 0176744 6431413& 167& 0176756 6124831& 168& 0176834 6127172 \\
169& 0176834 6141241& 170& 0176843 6116413& 171& 0176847 6121632 \\
172& 0176865 6121730& 173& 0176873 6121742& 174& 0177663 6431413 \\
175& 0177683 6128280& 176& 0177686 6145161& 177& 0178262 6116271 \\
178& 0178283 6141231& 179& 0178325 6142411& 180& 0178346 6142420 \\
181& 0178356 6124281& 182& 0178365 6141440& 183& 0178365 6431310 \\
184& 0178377 6183422& 185& 0178384 6121641& 186& 0178683 6112422 \\
187& 0178686 6123123& 188& 0178737 6118273& 189& 0182357 6126411 \\
190& 0182652 6124842& 191& 0182655 6431311& 192& 0182657 6141222 \\
193& 0182662 6127481& 194& 0182663 6122863& 195& 0182664 6124263 \\
196& 0182665 6124712& 197& 0182668 6141522& 198& 0182761 6127773 \\
199& 0182766 6141223& 200& 0182883 6121763& 201& 0183241 6184683 \\
202& 0183262 6126473& 203& 0183277 6182231& 204& 0183521 6124681 \\
205& 0183533 6141650& 206& 0183553 6121641& 207& 0183557 6127130 \\
208& 0183624 6142283& 209& 0183644 6142281& 210& 0183657 6122481 \\
211& 0183734 6128421& 212& 0183734 6142240& 213& 0183744 6142411 \\
214& 0183753 6124411& 215& 0183757 6141232& 216& 0183767 6124233 \\
217& 0183773 6141272& 218& 0183774 6121642& 219& 0183776 6124142 \\
220& 0186553 6132311& 221& 0186652 6131422& 222& 0186774 6132312 \\
223& 0186827 6131223& 224& 0186847 6131222& 225& 0187626 6132281 \\
226& 0187667 6131422& 227& 0611364 1662871& 228& 0611453 1648653 \\
229& 0611455 1648422& 230& 0611546 1626832& 231& 0611554 1627841 \\
232& 0611645 1643851& 233& 0611654 1286740& 234& 0611686 1286320 \\
235& 0611763 1624852& 236& 0611764 1642733& 237& 0611853 1626233 \\
238& 0611853 1642523& 239& 0611874 1638220& 240& 0612286 1663480 \\
\hline
\end{tabular} \\
\end{center}

\begin{center}
\begin{tabular}{rl|rl|rl}
\multicolumn{6}{c}{Table 7: (continued)} \\ \hline 
\multicolumn{1}{c}{} & \multicolumn{1}{c}{$AB$ \quad $CD$} & 
\multicolumn{1}{c}{} & \multicolumn{1}{c}{$AB$ \quad $CD$} & 
\multicolumn{1}{c}{} & \multicolumn{1}{c}{$AB$ \quad $CD$} \\ \hline
241& 0612356 1648241& 242& 0612457 1662171& 243& 0612517 1642753 \\
244& 0612528 1638242& 245& 0612536 1287762& 246& 0612547 1624820 \\
247& 0612548 1644511& 248& 0612556 1286422& 249& 0612585 1623820 \\
250& 0612586 1286420& 251& 0612645 1286441& 252& 0612646 1624730 \\
253& 0612744 1642280& 254& 0612752 1624282& 255& 0612752 1624670 \\
256& 0612753 1644131& 257& 0612753 1661181& 258& 0612758 1624250 \\
259& 0612764 1286620& 260& 0612844 1661271& 261& 0612853 1626142 \\
262& 0613515 1663843& 263& 0613554 1628243& 264& 0613647 1628323 \\
265& 0613753 1628263& 266& 0613753 1664122& 267& 0614236 1664450 \\
268& 0614271 1663481& 269& 0614274 1662142& 270& 0614475 1628213 \\
271& 0614527 1622843& 272& 0614571 1286372& 273& 0614571 1286671 \\
274& 0614577 1643513& 275& 0614638 1286321& 276& 0614642 1286341 \\
277& 0614675 1623412& 278& 0614681 1623233& 279& 0614773 1643450 \\
280& 0614774 1643270& 281& 0614784 1624250& 282& 0614873 1643211 \\
283& 0614884 1645130& 284& 0615118 1268632& 285& 0615146 1268462 \\
286& 0615147 1268452& 287& 0615471 1276443& 288& 0615487 1268111 \\
289& 0615627 1612342& 290& 0615741 1263482& 291& 0615743 1616242 \\
292& 0615783 1228631& 293& 0615784 1271660& 294& 0615785 1614350 \\
295& 0615874 1612341& 296& 0615883 1262313& 297& 0616234 1264760 \\
298& 0616237 1264280& 299& 0616281 1228632& 300& 0616358 1267113 \\
301& 0616472 1614551& 302& 0616481 1262343& 303& 0616481 1614232 \\
304& 0616534 1261732& 305& 0616841 1611822& 306& 0617413 1618861 \\
307& 0617426 1268423& 308& 0617526 1616431& 309& 0617556 1271642 \\
310& 0617556 1612741& 311& 0617572 1614651& 312& 0617586 1271660 \\
313& 0617625 1262433& 314& 0617626 1262343& 315& 0617626 1614232 \\
316& 0617655 1613422& 317& 0617682 1613640& 318& 0617786 1271642 \\
319& 0617786 1612741& 320& 0617824 1264620& 321& 0617824 1616231 \\
322& 0617844 1616450& 323& 0617845 1266212& 324& 0617853 1612461 \\
325& 0617853 1614222& 326& 0617874 1271662& 327& 0617884 1262450 \\
328& 0618516 1262443& 329& 0618517 1262433& 330& 0618527 1612441 \\
331& 0618557 1613450& 332& 0618616 1613441& 333& 0618625 1613441 \\
334& 0618714 1612662& 335& 0618748 1228631& 336& 0618824 1264142 \\
337& 0618825 1263172& 338& 0618847 1612361& 339& 0618874 1613241 \\
340& 0618883 1613451& 341& 0631147 1686422& 342& 0631352 1686222 \\
343& 0631458 1682421& 344& 0631557 1681411& 345& 0631778 1681411 \\
346& 0634135 1686220& 347& 0634187 1682321& 348& 0635124 1644340 \\
349& 0635132 1644270& 350& 0635137 1638220& 351& 0635138 1626240 \\
352& 0635152 1624740& 353& 0635311 1645860& 354& 0635381 1626710 \\
355& 0635441 1646161& 356& 0635857 1621341& 357& 0636178 1634410 \\
358& 0636178 1641321& 359& 0636882 1631443& 360& 0638171 1644131 \\
\hline
\end{tabular} \\
\end{center}

\begin{center}
\begin{tabular}{rl|rl|rl}
\multicolumn{6}{c}{Table 7: (continued)} \\ \hline 
\multicolumn{1}{c}{} & \multicolumn{1}{c}{$AB$ \quad $CD$} & 
\multicolumn{1}{c}{} & \multicolumn{1}{c}{$AB$ \quad $CD$} & 
\multicolumn{1}{c}{} & \multicolumn{1}{c}{$AB$ \quad $CD$} \\ \hline
361& 0638171 1661181& 362& 0638177 1286420& 363& 0638188 1286620 \\
364& 0638221 1645681& 365& 0638237 1642620& 366& 0638248 1642470 \\
367& 0638281 1644141& 368& 0638285 1643440& 369& 0638372 1624821 \\
370& 0638384 1661242& 371& 0638427 1638221& 372& 0638428 1638421 \\
373& 0638457 1624420& 374& 0638487 1626132& 375& 0638522 1633822 \\
376& 0638744 1632281& 377& 0638747 1641272& 378& 0641147 1638422 \\
379& 0641163 1286372& 380& 0641163 1286671& 381& 0641264 1643430 \\
382& 0641278 1624260& 383& 0641361 1624842& 384& 0641361 1643833 \\
385& 0641451 1286372& 386& 0641451 1286671& 387& 0641457 1622453 \\
388& 0641517 1632482& 389& 0641712 1632783& 390& 0641754 1621662 \\
391& 0641867 1621370& 392& 0643113 1648422& 393& 0643412 1662713 \\
394& 0643536 1286881& 395& 0643545 1661421& 396& 0643611 1624573 \\
397& 0643814 1638221& 398& 0643836 1661222& 399& 0643842 1661271 \\
400& 0643843 1646150& 401& 0644124 1624570& 402& 0644134 1627611 \\
403& 0644145 1286240& 404& 0644381 1626122& 405& 0644813 1621263 \\
406& 0644813 1634121& 407& 0644823 1634141& 408& 0645172 1616441 \\
409& 0645387 1263821& 410& 0645651 1226343& 411& 0645857 1261243 \\
412& 0645871 1261623& 413& 0646123 1261373& 414& 0646138 1261622 \\
415& 0646387 1261233& 416& 0648121 1264760& 417& 0648175 1262282 \\
418& 0648176 1614430& 419& 0648225 1263482& 420& 0648272 1612662 \\
421& 0648275 1614460& 422& 0648276 1262282& 423& 0648376 1264620 \\
424& 0648376 1616231& 425& 0648412 1264523& 426& 0648426 1263323 \\
427& 0648472 1264141& 428& 0648481 1264161& 429& 0648586 1226480 \\
430& 0648587 1226380& 431& 0648635 1226481& 432& 0648863 1226453 \\
433& 0655416 1126471& 434& 0655416 1162182& 435& 0655417 1126371 \\
436& 0658125 1164341& 437& 0658135 1166450& 438& 0658146 1162461 \\
439& 0658147 1162451& 440& 0658173 1127621& 441& 0658275 1126760 \\
442& 0658364 1162461& 443& 0658463 1127762& 444& 0658487 1127631 \\
445& 0661274 1127641& 446& 0663582 1164522& 447& 0663614 1163441 \\
448& 0663875 1163241& 449& 0663885 1163422& 450& 0682412 1186822 \\
451& 0685236 1128263& 452& 0685424 1128623& 453& 0685536 1162740 \\
454& 0685637 1126370& 455& 0685724 1162323& 456& 0685753 1126341 \\
457& 0685827 1162742& 458& 0685846 1162762& 459& 0685863 1126432 \\
460& 0686142 1163422& 461& 0686154 1122671& 462& 0686213 1162471 \\
463& 0686215 1163322& 464& 0686253 1126452& 465& 0686273 1126760 \\
466& 0686357 1163222& 467& 0686374 1164650& 468& 0686413 1164560 \\
469& 0686424 1162731& 470& 0686451 1126341& 471& 0687515 1162451 \\
472& 0687525 1162631& 473& 0687561 1126361& 474& 0687645 1126760 \\
475& 0688763 1164561&    &                &    &                 \\ 
\hline
\end{tabular} \\
\end{center}

\end{document}